\documentclass[11pt]{article}
\usepackage{amsmath}
\usepackage{amsthm}
\usepackage{amssymb,amsfonts}
\usepackage{graphicx}
\usepackage{hyperref}
\usepackage{latexsym}
\usepackage{todonotes}
\usepackage{nicefrac}
\usepackage{epsfig}
\usepackage{stmaryrd}
\usetikzlibrary{calc}
\usepackage{setspace}
\usepackage{enumerate}
\usepackage[all]{xypic}
\usepackage{bbm,ifpdf}
\ifpdf
\usepackage{pdfsync}
\fi

\newtheorem{theorem}{Theorem}[section]

\newtheorem{lemma}[theorem]{Lemma}
\newtheorem{proposition}[theorem]{Proposition}
\newtheorem{corollary}[theorem]{Corollary}

{
\theoremstyle{definition}

\newtheorem{remark}[theorem]{Remark}
}

\newcommand{\excise}[1]{}

\renewcommand{\dim}{\operatorname{dim}}

\newcommand{\crk}{\operatorname{crk}}

\renewcommand{\and}{\qquad\text{and}\qquad}
\newcommand{\Ind}{\operatorname{Ind}}

\newcommand{\Hom}{\operatorname{Hom}}

\newcommand{\Z}{\mathbb{Z}}

\newcommand{\C}{\mathbb{C}}

\renewcommand{\H}{\operatorname{H}}
\newcommand{\IH}{\operatorname{IH}}

\newcommand{\cA}{\mathcal{A}}
\newcommand{\la}{\lambda}

\newcommand{\cs}{\C^\times}

\newcommand{\nicktodo}{\todo[inline,color=green!20]}

\begin{document}
\spacing{1.2}
\noindent{\Large\bf Intersection cohomology of the symmetric reciprocal plane}\\

\noindent{\bf Nicholas Proudfoot}\footnote{Supported by NSF grant DMS-0950383.}\\
Department of Mathematics, University of Oregon,
Eugene, OR 97403\\

\noindent{\bf Max Wakefield}\footnote{Supported by the Simons Foundation, the Office of Naval Research,
and the Japan Society for the Promotion of Science.}\\
Department of Mathematics, United States Naval Academy, Annapolis, MD 21402\\

\noindent{\bf Ben Young}\\
Department of Mathematics, University of Oregon,
Eugene, OR 97403\\
{\small
\begin{quote}
\noindent {\em Abstract.}
We compute the Kazhdan-Lusztig polynomial of the uniform matroid of rank $n-1$ on $n$ elements
by proving that the coefficient of $t^i$ is equal to the number of ways 
to choose $i$ non-intersecting chords in an $(n-i+1)$-gon.  We also show that the corresponding
intersection cohomology group is isomorphic to the irreducible representation of $S_n$ 
associated to the partition $[n-2i,2,\ldots,2]$.
\end{quote} }

\section{Introduction}
Fix an integer $n\geq 2$.  Let
$$U_n := \big\{z\in (\cs)^{n}\mid \textstyle{\frac 1{z_1} + \ldots + \frac 1{z_n}} = 0\big\},$$
and let $X_n$ be the closure of $U_n$ in $\C^n$.
Equivalently, $X_n$ is the hypersurface defined by the $(n-1)^\text{st}$ elementary symmetric polynomial.
Since $X_n$ is isomorphic to the ``reciprocal plane" \cite{PS, SSV, DGS-OT}
associated to a symmetric configuration of $n$ vectors 
spanning a vector space of dimension $n-1$, 
we will refer to it as the {\bf symmetric reciprocal plane}.
The intersection cohomology of $X_n$ (with complex coefficients)
vanishes in odd degrees \cite[3.12]{EPW}, and we let
$c_{n,i} := \dim \IH^{2i}(X_n)$.
Let $$P_n(t) := \sum_i c_{n,i}\,t^i\and
\Phi(t,u) := \sum_{n=2}^\infty P_n(t)u^{n-1}.$$
The polynomial $P_n(t)$ is called the {\bf Kazhdan-Lusztig polynomial} of the uniform
matroid of rank $n-1$ on a set of $n$ elements.
It is clear that the collection of numbers $c_{n,i}$, the collection polynomials $P_n(t)$, and the single power series
$\Phi(t,u)$ all encode the same data.  
The following theorem gives three equivalent versions of a recursive procedure to compute these data
\cite[2.2, 2.19, 2.21, 3.12]{EPW}.\footnote{In \cite{EPW}, $c_{n,i}$ was denoted by $c_{1,n-1}^i$
and $P_n(t)$ was denoted by $P_{1,n-1}(t)$.}

\begin{theorem}\label{epw}
We have $c_{n,i} = 0$ for all $i \geq \frac{n-1}{2}$,
and the following (equivalent) equations hold.\\
\begin{enumerate}
\item[{\em (1)}] For all $n$ and $i$,
$\displaystyle c_{n,i} = (-1)^i\binom{n}{i} + 
\sum_{j=0}^{i-1}\sum_{k=2j+2}^{i+j+1}(-1)^{i+j+k+1}\binom{n}{k,i+j-k+1,n-i-j-1} c_{k,j}.$\\
\item[{\em (2)}] For all $n$,
$\displaystyle t^{n-1} P_{n}(t^{-1}) = \sum_{j=0}^{n-1} (-1)^j\binom{n}{j} \Big(t^{n-i-1}-1\Big) + \sum_{k=2}^{n} 
\binom{n}{k} (t-1)^{n-k} P_{k}(t).$\\
\item[{\em (3)}]
$\displaystyle\Phi(t^{-1}, tu) = \frac{tu-u}{(1-tu+u)(1+u)}
+ \frac{1}{(1-tu+u)^{2}} \;\Phi\!\left(t,\frac{u}{1-tu+u}\right).$\\
\end{enumerate}
\end{theorem}

The purpose of this paper is to give an explicit formula for $c_{n,i}$.  In fact, we do even better:
the intersection cohomology group $\IH^{2i}(X_n)$ is not just a vector space, but also a representation of
the symmetric group $S_n$, which acts on $X_n$ by permuting the coordinates.  We categorify the recursion
in Theorem \ref{epw}(1) to obtain a recursion in the virtual representation ring of the symmetric group, 
which we then solve as follows.
For any partition $\la$ of $n$, let $V_\la$ be the corresponding irreducible representation of $S_n$. 

\begin{theorem}\label{main}
For all $n\geq 2$ and $i\geq 0$, the following identities hold.\\
\begin{enumerate}
\item[{\em (1)}] $\displaystyle c_{n,i} = \frac{1}{i+1}\binom{n-i-2}{i}\binom{n}{i}$.\\
\item[{\em (2)}] $\IH^{2i}(X_n) \cong V_{[n-2i,2,\ldots,2]} = V_{[n-2i,2^i]}$.\\
\end{enumerate}
\end{theorem}

\begin{remark}\label{cayley}
Cayley~\cite{cayley1890} proved that the quantity $\frac{1}{i+1}\binom{n-i-2}{i}\binom{n}{i}$
is equal to the number of ways to choose $i$ non-intersecting chords in an $(n-i+1)$-gon.
This combinatorial interpretation will actually play a role in our proof of Theorem \ref{main}(1).
\end{remark}

\begin{remark}
Note that, if $i \geq \frac{n-1}{2}$, then the first binomial coefficient in Theorem \ref{main}(1) is zero and 
the ``partition" $[n-2i,2,\ldots,2]$ in Theorem \ref{main}(2)
is not a partition.  Thus the vanishing of these coefficients is built into the statement of the theorem.
\end{remark}

\begin{remark}
Theorem \ref{main}(1) follows immediately from Theorem \ref{main}(2) via the hook-length formula
for the dimension of $V_\la$.  However, we will need to prove Theorem \ref{main}(1) first in order
to prove Theorem \ref{main}(2).
\end{remark}

\begin{remark}\label{wow}
The most surprising aspect of Theorem \ref{main}(2) is that the 
representation $\IH^{2i}(X_n)$ is irreducible; we see no {\em a priori} reason for this to be the case.
One way to interpret Theorem \ref{main}(2) would be to regard it as a new geometric construction
of a certain class of irreducible representations of the symmetric group.
\end{remark}

\begin{remark}
In a future paper, we will compute the Kazhdan-Lusztig polynomial of a uniform matroid
of arbitrary rank on a set of $n$ elements.  The reason that we only treat the rank $n-1$ case 
in this paper is that only this case can be represented by an arrangement that carries
a natural action of the symmetric group.
\end{remark}

We conclude the introduction with a discussion of structure of the paper.
In Section \ref{sec:dim}, we use Theorem \ref{epw}(3) along with Beckwith's work on polygonal dissections \cite{Beckwith}
to prove Theorem \ref{main}(1).  As a corollary, we prove that the coefficients of $P_n(t)$ form
a log concave sequence, as conjectured in \cite[2.5]{EPW}.  
In Section \ref{sec:der}, we use mixed Hodge theory to construct a spectral sequence which allows us to lift
the recursion in Theorem \ref{epw}(1) to the category of representations of the symmetric group.
In Section \ref{sec:sol}, we use Schubert calculus to solve this recursion, thus proving Theorem \ref{main}(2).

\vspace{\baselineskip}
\noindent
{\em Acknowledgments:}
N.P. would like to thank Ben Webster for valuable conversations.
M.W. would like to thank Hokkaido University and the University of Oregon for their hospitality.
All three authors would like to acknowledge the On-Line Encyclopedia of Integer Sequences \cite{oeis},
without which this project would have been very difficult.  All computations, as well as substantial experimental work, were done with the assistance of the SAGE computer algebra system~\cite{sage}.

\section{Numbers}\label{sec:dim}
In this section we prove Theorem \ref{main}(1).
For all $n\geq 3$ and $i\geq 0$, let $d_{n,i}$ be the number of sets of $i$ non-intersecting diagonals of an $n$-gon,
and let $$f(t,u) := \sum_{\substack{i\geq 0\\ n\geq 3}}d_{n,i}\,t^iu^{n-1}.$$
As noted in Remark \ref{cayley}, Cayley~\cite{cayley1890} gave a formula for $d_{n,i}$ which demonstrates that
Theorem \ref{main}(1) is  equivalent to the statement $c_{n,i} = d_{n-i+1,i}$, which is in turn 
equivalent to the identity $\Phi(t,u) = u^{-1}f(tu,u)$.
Thus it will be sufficient to prove that $g(t,u) := u^{-1}f(tu,u)$ satisfies the functional equation in Theorem \ref{epw}(3).

Using the combinatorial interpretation of the numbers $d_{n,i}$, Beckwith \cite{Beckwith} shows that
$$f(t,u)= 
\frac{2 \, {\left({\left(2t + 1\right)} u + \sqrt{1-2 \, {\left(2t + 1\right)} u + u^{2}} - 1\right)}}{{1-\left(2t + 1\right)}^{2}}, 
$$
and therefore 
$$g(t,u)= \frac{2}{u}\cdot
\frac{{\left(2tu + 1\right)} u + \sqrt{1-2 \, {\left(2tu + 1\right)} u + u^{2}} - 1}{{1-\left(2tu + 1\right)}^{2}}.
$$
It is straightforward to show (or to check on a computer) that
\begin{eqnarray*}
g(t^{-1},tu) &=& \frac{1-tu-2tu^2-\sqrt{1-2tu -4t u^2 + t^2u^2}}{2t u^2(1+u)}\\
&=& \frac{1-tu-2tu^2-(1-tu+u)\sqrt{\frac{1-2tu -4t u^2 + t^2u^2}{(1-tu+u)^2}}}{2t u^2(1+u)}\\
&=& \frac{tu-u}{(1-tu+u)(1+u)}
+ \frac{1}{(1-tu+u)^{2}} \;g\!\left(t,\frac{u}{1-tu+u}\right).
\end{eqnarray*}
This completes the proof of Theorem \ref{main}(1).\qed\\

In a previous paper, Elias and the first two authors conjectured 
that the coefficients of the Kazhdan-Lusztig polynomial of any matroid
form a log concave sequence \cite[2.5]{EPW}.  Theorem \ref{main}(1) allows us to prove this conjecture for the
uniform matroid of rank $n-1$ on $n$ elements.

\begin{corollary}\label{log concave}
Fix an integer $n\geq 2$.  The sequence
$c_{n,0}, c_{n,1},\ldots,c_{\lfloor n/2\rfloor - 1}$ is strictly log concave.
That is, for all $0<i<\lfloor n/2\rfloor - 1$, we have
$$c_{n,i}^2 > c_{n,i-1}c_{n,i+1}.$$
\end{corollary}

\begin{proof}
Fix such an $n$ and $i$.
By Theorem \ref{main}(1), we have
$$\frac{c_{n,i-1}c_{n,i+1}}{c_{n,i}^2} = 
\frac{i}{i+2}\cdot\frac{n-i-1}{n-i+1}\cdot\frac{n-2i-2}{n-2i}\cdot\frac{n-2i-3}{n-2i-1}\cdot\frac{n-i}{n-i-2}.$$
The first four of these factors are clearly each less than 1, while the fifth factor is greater than 1.
However, if we combine the fourth and fifth factors, we will find that their product is less than 1.
Indeed, for any $0\leq k<\ell$, we have $$\frac{k}{k+2}\cdot\frac{\ell+2}{\ell}<1.$$
Applying this to $k=n-2i-3$ and $\ell = n-i-2$,
we see that the product of the fourth and fifth factors is less than 1, thus so is the entire expression.
\end{proof}

\section{Deriving the categorified recursion}\label{sec:der}
Our goal in this section is to promote the recursion in Theorem \ref{epw}(1) to the level of virtual
representations of symmetric groups, which we will use in the next section to prove Theorem \ref{main}(2).
For any subset $S\subset \{1,\ldots,n\}$, let $Y(S)\subset X_n$ be the locus of points whose
vanishing coordinates coincide exactly with $S$. 

\excise{
If $S\neq \emptyset$, let $i_S$ be its minimal element, and
let $W_n(S)\subset X_n$ be the locus of points $(z_1,\ldots,z_n)$
satisfying the following two conditions:
\begin{itemize}
\item if $z_i=0$ then $i\in S$
\item $z_{i_S} \sum_{i\notin S}\frac{1}{z_i}\neq 1$.
\end{itemize}
Then $W_n(S)$ is open in $X_n$ and $Y_n(S)$ is closed in $W_n(S)$.
}

\begin{lemma}\label{nbhd}
If $|S|>1$,
then there is an open neighborhood of $Y(S)$ in $X_n$ that is
isomorphic to an open neighborhood of $Y(S)\cong Y(S)\times\{0\}$ in $Y(S)\times X_{|S|}$.
Furthermore, the isomorphism may be chosen to restrict to the identity map on $Y(S)$.
\end{lemma}

\begin{proof}
Let $W(S)\subset X_n$ be the locus defined by the nonvanishing of $z_i$ for all $i\notin S$.
Then $W(S)$ is open in $X_n$ and $Y(S)$ is closed in $W(S)$.
Assume without loss of generality that $S = \{1,\ldots,r\}$ for some $r>1$.
Then
$$\C[W(S)] = \C[x_1,\ldots,x_r,x_{r+1}^\pm,\ldots,x_n^\pm]\Big{/}
\left\langle \sum_{i=1}^n \prod_{j\neq i} x_j\right\rangle$$
and, since $Y(S)\cong (\cs)^{n-r}$ \cite[Proposition 5]{PS},
$$\C[Y(S)\times X_{|S|}] \cong \C[x_1,\ldots,x_r,x_{r+1}^\pm,\ldots,x_n^\pm]
\Big{/}\left\langle \sum_{i=1}^{r} \prod_{r\geq j\neq i} x_j\right\rangle.$$
In both cases, the subvariety $Y(S)$ is defined by the vanishing of $x_1,\ldots,x_r$.
Consider the open subset $V(S)\subset W(S)$ defined by inverting $1+x_1\sum_{k>r}x_k^{-1}$,
along with the open subset 
$U(S)\subset Y(S)\times X_{|S|}$ defined by inverting $1-x_1\sum_{k>r}x_k^{-1}$.
It is clear that both of these open subsets contain $Y(S)$.
Now consider the maps
$$\varphi:V(S)\rightleftarrows U(S):\psi$$
given by the formulas
$$\varphi^*(x_1) := \frac{x_1}{1+x_1\displaystyle{\sum_{k>r}}x_k^{-1}}\and\varphi(x_i) = x_i\;\;\text{if $i>1$}$$
and
$$\psi^*(x_1) := \frac{x_1}{1-x_1\displaystyle{\sum_{k>r}}x_k^{-1}}\and\psi(x_i) = x_i\;\;\text{if $i>1$}.$$
These two maps clearly restrict to the identity on $Y(S)$, and it is straightforward to check that they are mutually
inverse.
\end{proof}

Let $Y(p)\subset X_n$ be the union of strata of codimension $p$.  That means that $Y(p) = \bigsqcup_{|S|=p+1}Y(S)$
if $p>0$ and $Y(0) = Y(\emptyset)$.
Let $\iota_p:Y(p)\hookrightarrow X_n$ be the inclusion.  The following lemma tells us how to compute
the hypercohomology of the shriek pullback to $Y(p)$ of the intersection cohomology sheaf of $X_n$,
which we will need in the proof of Proposition \ref{spectral sequence}.

\begin{lemma}\label{triv-ls}
For any $p>0$, we have $\mathbb{H}^{*}(\iota_p^!\operatorname{IC}_{X_n}) 
\cong \H^*(Y(p))\otimes \IH^*_c(X_{p+1})$.
\end{lemma}

\begin{proof}
The cohomology of the complex $\iota_p^!\operatorname{IC}_{X_n}$ is a local system on $Y(p)$
whose fiber at a point is the compactly supported cohomology of the stalk of $\operatorname{IC}_{X_n}$
at that point.  By Lemma \ref{nbhd}, this local system is constant with fiber isomorphic to $\IH^*_c(X_{p+1})$.
We thus have a spectral sequence with $$E_2^{j,k} = \H^j(Y(p))\otimes \IH^k_c(X_{p+1})
\and\bigoplus_{j+k=\ell}E_\infty^{j,k} = \mathbb{H}^{*}(\iota_p^!\operatorname{IC}_{X_n})\;\;\text{for all $\ell$}.$$
All of these groups carry mixed Hodge structures, and the maps in the spectral sequence
are strictly compatible with the weight filtrations.
The group $H^j(Y(p))$ is pure of weight $2j$ \cite{Shapiro}, and $\IH^{k}_c(X_{p+1})$ 
is pure of weight $k$ \cite[3.9]{EPW}.
Thus all maps vanish and the result follows.
\end{proof}

Let $\rho_n \cong V_{[n]} \oplus V_{[n-1,1]}$ be the permutation representation of $S_n$.

\begin{proposition}\label{spectral sequence}
Fix integers $n\geq 2$ and $i<\frac{n-1}{2}$.
There exists a first quadrant cohomological spectral sequence $E$ in the category of $S_n$-representations
with
$$E_1^{p,q} = 
\begin{cases}
\Ind_{S_{n-p-1}\times S_{p+1}}^{S_n}
\left(\wedge^{2i-p-q} \rho_{n-p-1}\boxtimes \IH^{2(i-q)}(X_{p+1})\right)\;\;\text{if \;$0<p<n-1$,}\\
\wedge^i\rho_n\;\;\text{if \;$p=0$ and $q=i$},\\
0\;\;\text{otherwise}
\end{cases}
$$
converging to
$$\bigoplus_{p+q=2i}E_\infty^{p,q} = \IH^{2i}(X_n)\and \bigoplus_{p+q\neq 2i}E_\infty^{p,q} = 0.$$
\end{proposition}

\begin{proof}
There is a first quadrant cohomological spectral sequence $\tilde E$ with
$$\tilde E_1^{p,q} = \mathbb{H}^{p+q}(\iota_p^!\operatorname{IC}_{X_n})
\and
\bigoplus_{p+q=\ell}\tilde E_\infty^{p,q} = \IH^\ell(X_n)\;\;\text{for all $\ell$ \cite[\S 3.4]{BGS96}}.$$
Lemma \ref{triv-ls} tells us that
$$\tilde E_1^{p,q} \;\;\cong
\bigoplus_{j+2k=p+q} \H^j(Y(p))\otimes \IH^{2k}_c(X_{p+1})\quad\text{if $p>0$}.$$
If $p=0$, then $\iota_p$ is an open inclusion inside of the smooth locus of $X_n$,
so
$\tilde E_1^{0,q}\cong \H^q(Y(0)).$
\excise{
The sheaf $\iota_p^!\operatorname{IC}_{X_n}$ is a local system whose stalks are isomorphic 
to the compactly supported cohomology of the normal slice to $Y(p)$.

If $p=n-1$, then $Y(p)$ is a point.  If $p=0$, then $Y(p)\subset X_n$ is open, and $\iota_p^!\operatorname{IC}_{X_n}$
is the constant sheaf.  Let us now assume that $0<p<n-1$, and choose a subset $S$ of size $p$.
A local system on $Y(S)$ with stalks isomorphic to $\IH^{*}_c(X_{p+1})$ is the same as an action of 
the group $\pi_1(Y(S))\cong \Z^{n-1-p}$ on $\IH^{*}_c(X_{p+1})$.  Since $Y(S)$ is acyclic, the cohomology
of this local system is simply the group cohomology of $\Z^{n-1-p}$ with coefficients in $\IH^{*}_c(X_{p+1})$,
which is in turn equal to the cohomology of a certain differential on the vector space $\H^*(Y(S))\otimes \IH^{*}_c(X_{p+1})$.
Thus we may extend our spectral sequence to one with
$$\tilde E_0^{p,q}\;\;\cong
\bigoplus_{j+2k=p+q} \H^j(Y(p))\otimes \IH^{2k}_c(X_{p+1}),$$
where the differential on the $\tilde E_0$ page (increasing $j$ and $q$, fixing $k$ and $p$)
is the aforementioned differential that is used to compute the cohomology of a local system.
}

As in the proof of Lemma \ref{triv-ls}, our groups carry mixed Hodge structures, 
and the maps are strictly compatible with weight filtrations.
As noted above, $H^j(Y(p))$ is pure of weight $2j$, while $\IH^{2k}(X_n)$ and $\IH^{2k}_c(X_{p+1})$ 
are both pure of weight $2k$.  
Let $E$ be the maximal subquotient of $\tilde E$ that is pure of weight $2i$.
Then 
$$E_1^{0,i} = \H^i(Y(0))\and E_1^{0,q} = 0\quad\text{if $q\neq i$},$$
and
$$\bigoplus_{p+q=2i}E_\infty^{p,q} = \IH^{2i}(X_n)\and \bigoplus_{p+q\neq 2i}E_\infty^{p,q} = 0.$$
If $p>0$,
we have
\begin{eqnarray*}E_1^{p,q}\;\;&\cong&
\bigoplus_{\substack{j+2k=p+q\\ j+k=i}} \H^j(Y(p))\otimes \IH^{2k}_c(X_{p+1})\\
&\cong& \H^{2i-p-q}(Y(p)) \otimes \IH^{2(p+q-i)}_c(X_{p+1})\\\\
&\cong& 
\H^{2i-p-q}(Y(p)) \otimes\IH^{2(i-q)}(X_{p+1})\quad \text{by Poincar\'e duality \cite[1.4.6.5]{Conrad-etale}}.
\end{eqnarray*}
Note that when $p=n-1$, $Y(p)$ is a point, so $\H^{2i-p-q}(Y(p)) = 0$ unless $q=2i-n+1$,
but then $i-q > \frac{n-1}{2}$, so $\IH^{2(i-q)}(X_{p+1}) = 0$.
Thus $E_1^{n-1,q} = 0$ for all $q$.

Last, we incorporate the action of $S_n$.  For each $p>0$, $S_n$ acts transitively on the set of components of $Y(p)$
with stabilizer $S_{n-p-1}\times S_{p+1}$,
thus 
$$E_1^{p,q} \cong \Ind_{S_{n-p-1}\times S_{p+1}}^{S_n}\!\Big(\H^{2i-p-q}(Y([p])) \otimes\IH^{2(i-q)}(X_{p+1})\Big).$$
Since $p>0$, we have $Y([p])\cong (\cs)^{n-p-1}$, with $S_{n-p-1}$ permuting the coordinates and $S_{p+1}$ acting
trivially, thus $\H^{2i-p-q}(Y([p])) \cong \wedge^{2i-p-q} \rho_{n-p-1}$ as a representation of $S_{n-p-1}$.
The action of $S_{n-p-1}\times S_{p+1}$ factors through $S_{p+1}$, which is the stabilizer of a single point in $Y([p])$.
Finally, since $Y(0)$ is a generic $(n-1)$-dimensional linear slice of $(\cs)^n$
and $i<\frac{n-1}{2}<n$,
we have $E_1^{0,i} = \H^i(Y(0)) \cong \wedge^i\rho_n$.
\end{proof}
\excise{
\nicktodo{I'm worried about the statement that the spectral sequence respects weights.
If that's really true, then the maps on the $E_0$ page have to be trivial.  This may be, but it seems
unreasonable to assume.  (Assuming that weights are respected on the $E_1$ page and beyond
seems reasonable.)  I think that, by looking closely at the \'etale local slices, we can conclude that we
can work in a neighborhood that sees the whole fundamental group, and therefore the local system is
trivial.}
}

\begin{corollary}\label{virtual}
We have $$\IH^{2i}(X_n) \;\;\cong\;\; (-1)^i \wedge^i\rho_n \:\;\;\oplus\!\! 
\bigoplus_{\substack{0<p<n-1\\ 0\leq q\leq \min(i,2i-p)}} (-1)^{p+q}\; 
\operatorname{Ind}_{S_{n-p-1}\times S_{p+1}}^{S_{n}}\left(\wedge^{2i-p-q}\rho_{n-p-1}
\boxtimes\IH^{2(i-q)}(X_{p+1})\right)$$
as virtual representations of $S_n$.
\end{corollary}

\begin{proof}
Given a finite convergent spectral sequence in a semisimple category, one always has the Euler characteristic
equation $$\bigoplus_{p,q}\; (-1)^{p+q}\; E_\infty^{p,q} \;\;\cong\;\; \bigoplus_{p,q}\; (-1)^{p+q}\; E_1^{p,q}.$$
The result then follows from Proposition \ref{spectral sequence}.
\end{proof}

\begin{remark}
Taking dimensions in Corollary \ref{virtual}, we obtain the integer equation
\begin{eqnarray*}c_{n,i} \;\;&=&\;\; (-1)^i\binom{n}{i} \;\;\;\;+ \!\sum_{\substack{0<p<n-1\\ 0\leq q\leq \min(i,2i-p)}}\! (-1)^{p+q}
\binom{n}{n-p-1}\binom{n-p-1}{2i-p-q} c_{p+1,i-q}\\
&=&\;\; (-1)^i\binom{n}{i} \;\;\;\;+ \!\sum_{\substack{0<p<n-1\\ 0\leq q\leq \min(i,2i-p)}}\! (-1)^{p+q}
\binom{n}{p+1, 2i-p-q, n+q-2i-1} c_{p+1,i-q}.\end{eqnarray*}
If we make the substitution $k=p+1$ and $j=i-q$, we obtain the recursion in Theorem \ref{epw}(1).
Thus Corollary \ref{virtual} is indeed a categorification of Theorem \ref{epw}.
\end{remark}

\begin{remark}
It is clear how to generalize Corollary \ref{virtual} to arbitrary arrangements.
Using the notation of \cite[\S 3.1]{EPW},
let $\cA$ be an arrangement, let $U_{\!\cA}$ be the complement of $\cA$, 
and let $X_\cA$ denote the reciprocal plane of $\cA$.
Let $L$ be the lattice of flats; for each $F\in L$, we can define the localization $\cA_F$ and the restriction $\cA^F$.
Suppose that $W$ is a finite group that acts on $\cA$, and therefore on $L$ and on $X_\cA$.
For each $F\in L$, let $W(F)$ be its stabilizer, and let $[F]$ be its equivalence class
in $L/W$.  Then, in the virtual representation ring of $W$, we have
$$\IH^{2i}(X_\cA) \;\;\cong \bigoplus_{\substack{[F]\in L/W\\ j\geq 0}} (-1)^j\; \Ind_{W(F)}^W\!\Big(H^j(U_{\!\cA_F})\otimes \IH^{2(\crk F - i + j)}(X_{\cA^F})\Big).$$
In the case of Corollary \ref{virtual}, $W(F)$ decomposes as a product of two groups, one of which
acts only on $H^j(U_{\!\cA_F})$ and the other of which acts only on $\IH^{2(\crk F - i + j)}(X_{\cA^F})$,
which greatly simplifies the formula from a computational standpoint.  This phenomenon does not occur
in general; for example, it does not occur for the symmetric group acting on the braid arrangement,
where flats are indexed by set-theoretic partitions.
\end{remark}

\section{Solving the categorified recursion}\label{sec:sol}
In this section we prove Theorem \ref{main}(2).  We proceed by induction on $n$.
When $n=2$, we must have $i=0$, and the theorem says that $\IH^0(X_2)$ is equal to the trivial
representation $V_{[2]}$ of $S_2$.  Since intersection cohomology agrees with ordinary cohomology
in degree 0, this is obvious.  Now let $n>2$ be given and assume that $\IH^{2i}(X_m) \cong V_{[m-2i,2^i]}$
for all $2\leq m<n$ and $i<\frac{m-1}{2}$.

\begin{lemma}\label{key}
Suppose that $0\leq i<\frac{n-1}{2}$, $0<p<n$, and $0\leq q\leq\min(i,2i-p)$.
\begin{itemize}
\item[{\em (i)}] $\Hom_{S_n}\!\Big(V_{[n-2i,2^i]},\; \wedge^i\rho_n\Big) \neq 0$ if and only if $i=0$.
\item[{\em (ii)}] $\Hom_{S_n}\!\left(V_{[n-2i,2^i]},\; \operatorname{Ind}_{S_{n-p-1}\times S_{p+1}}^{S_{n}}
\!\left(\wedge^{2i-p-q}\rho_{n-p-1}
\boxtimes\IH^{2(i-q)}(X_{p+1})\right)\right)\neq 0$

\noindent
if and only if $i>0$, $p=2i-1$, and $q=1$.
\end{itemize}
\end{lemma}

\begin{proof}
We begin with part (i).  If $i=0$, then $V_{[n-2i,2^i]} = V_{[n]}$ and $\wedge^i\rho_n$ are both equal to the 1-dimensional
trivial representation.  If $i>0$, then 
$$
\wedge^i\rho_n \cong \wedge^i (V_{[n]} \oplus V_{[n-1,1]})
\cong \bigoplus_{j+k=i} \wedge^jV_{[n]} \otimes \wedge^kV_{[n-1,1]}.$$
But $V_{[n]}$ is trivial, so $\wedge^jV_{[n]}$ is trivial if $j\in\{0,1\}$ and zero otherwise,
and $\wedge^kV_{[n-1,1]}\cong V_{[n-k,1^k]}$ \cite[Exercise 4.6]{FH},
thus $$\wedge^i\rho_n\cong \wedge^iV_{[n-1,1]}\oplus \wedge^{i-1}V_{[n-1,1]}
\cong V_{[n-i,1^i]}\oplus V_{[n-i+1,1^{i-1}]}.$$
In particular, it does not contain $V_{[n-2i,2^i]}$ as a summand.

We now move on to part (ii).  By the argument above, we have 
$$\wedge^{2i-p-q}\rho_{n-p-1}\cong V_{[n+q-2i-1,1^{2i-p-q}]} \oplus V_{[n+q-2i,1^{2i-p-q-1}]}$$
as representations of $S_{n-p-1}$
(here we adhere to the convention that if the subscript is not a partition,
the corresponding representation is zero).  By our inductive hypothesis,
$$\IH^{2(i-q)}(X_{p+1})\cong V_{[p+2q-2i+1,2^{i-q}]}$$
as representations of $S_{p+1}$.
Thus we are interested in
$$\Hom_{S_n}\!\Big(V_{[n-2i,2^i]},\;\operatorname{Ind}_{S_{n-p-1}\times S_{p+1}}^{S_{n}}
\left(V_{[n+q-2i-1,1^{2i-p-q}]}\boxtimes V_{[p+2q-2i+1,2^{i-q}]}\right)\Big)$$
$$\oplus$$
$$\Hom_{S_n}\!\Big(V_{[n-2i,2^i]},\;\operatorname{Ind}_{S_{n-p-1}\times S_{p+1}}^{S_{n}}
\left(V_{[n+q-2i,1^{2i-p-q-1}]}\boxtimes V_{[p+2q-2i+1,2^{i-q}]}\right)\Big).$$

The dimension $c^\nu_{\mu\la}$ of 
$$\Hom_{S_{|\nu|}}\!\Big(V_\nu, \;\operatorname{Ind}_{S_{|\mu|}\times S_{|\la|}}^{S_{\nu}}
\left(V_{\mu}\boxtimes V_{\la}\right)\Big)$$
is equal to the number of Littlewood-Richardson tableaux of shape $\nu/\la$ and weight $\mu$ \cite[\S 5.2]{YoungTableaux}.
We will let $\nu = [n-2i,2^i]$, $\mu = [n+q-2i-1,1^{2i-p-q}]$, $\mu' = [n+q-2i,1^{2i-p-q-1}]$, and
$\la = [p+2q-2i+1,2^{i-q}]$, and we will compute the coefficients $c^\nu_{\mu\la}$ and $c^\nu_{\mu'\la}$.
In order for either of these two coefficients to be nonzero, we need the diagram of shape $\la$ to fit
inside of the diagram of shape $\nu$, which is the case if and only if $n\geq p+2q+1$.  
In this case, the skew diagram $\nu/\la$ consists of a $1\times (n-p-2q-1)$ component in the top-right
and a $q\times 2$ component in the bottom-left.  Below is a picture of $\nu/\la$
with $n=20$, $i=6$, $p=8$, and $q=3$.\\

\begin{center}
\includegraphics[width=4cm]{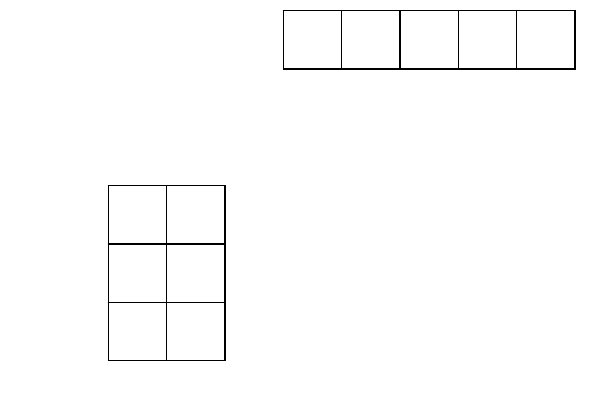}
\rule{1cm}{0cm}
\end{center}


\noindent
Any Littlewood-Richardson tableau of this shape must consist of all 1's in the upper-right and two
ascending sequences of length $q$ in the lower-left, beginning (in the top row) with either 1 or 2.  
Since $\mu$ and $\mu'$ are both hooks, it is not possible
for a tableau of weight $\mu$ or $\mu'$ to contain exactly two 2's, thus we may assume that $q\in\{0,1\}$
and the content of our tableau consists of at most one 2 and all remaining entries equal to 1.  
If the weight is $\mu$, this means that $2i-p-q=0$ or $1$, so $\la = [q+1,2^{i-q}]$ or $[q,2^{i-q}]$.
For $\la$ to be a partition and $p$ to be positive, we must have $i>0$, $p=2i-1$, and $q=1$.
Thus we may conclude that $c^\nu_{\mu\la} = 1$ if
$i>0$, $p=2i-1$, and $q=1$, and it is zero otherwise.
By similar reasoning, we deduce that $c^\nu_{\mu'\la}$ is always equal to zero.
\end{proof}

We now complete the proof of Theorem \ref{main}(2).
By Lemma \ref{key}, there is exactly one term on the right-hand side of the isomorphism in Corollary \ref{virtual}
that contains $V_{[n-2i,2^i]}$ as a summand.
This implies that $\IH^{2i}(X_n)$ contains $V_{[n-2i,2^i]}$ as a summand.
But Theorem \ref{main}(1), along with the hook-length formula for the dimension of $V_{[n-2i,2^i]}$,
tells us that $\dim \IH^{2i}(X_n) = \dim V_{[n-2i,2^i]}$.  Thus $\IH^{2i}(X_n)$ must be isomorphic
to $V_{[n-2i,2^i]}$.\qed

\bibliography{./symplectic}
\bibliographystyle{amsalpha}

\end{document}